%% file: DFO.tex
\documentclass{article}
\usepackage{microtype}
\usepackage{graphicx}
\usepackage{subfigure}
\usepackage{booktabs} 
\usepackage{amsmath}
\usepackage{amssymb}
\usepackage{amsthm}
\usepackage{mathtools}
\usepackage{comment}
\numberwithin{equation}{section}
\usepackage{wrapfig}
\usepackage{xfrac}
\usepackage{threeparttable}
\usepackage{algorithm}
\usepackage{algpseudocode}
\usepackage[colorlinks=true]{hyperref}
\usepackage[font=small]{caption}

\newtheorem{theorem}{Theorem}[section]
\newtheorem{assum}[theorem]{Assumption}
\newtheorem{lem}[theorem]{Lemma}

\newtheorem{remark}[theorem]{Remark}

\usepackage{xargs}                      
\usepackage[pdftex,dvipsnames]{xcolor}  
\usepackage[colorinlistoftodos,prependcaption,textsize=tiny]{todonotes}
\newcommandx{\unsure}[2][1=]{\todo[linecolor=red,backgroundcolor=red!25,bordercolor=red,#1]{#2}}
\newcommandx{\change}[2][1=]{\todo[linecolor=blue,backgroundcolor=blue!25,bordercolor=blue,#1]{#2}}
\newcommandx{\info}[2][1=]{\todo[linecolor=OliveGreen,backgroundcolor=OliveGreen!25,bordercolor=OliveGreen,#1]{#2}}
\newcommandx{\improvement}[2][1=]{\todo[linecolor=Plum,backgroundcolor=Plum!25,bordercolor=Plum,#1]{#2}}
\newcommandx{\thiswillnotshow}[2][1=]{\todo[disable,#1]{#2}}

\usepackage{array}
\usepackage{arydshln}
\usepackage{enumitem}

\PassOptionsToPackage{numbers, sort&compress}{natbib}

\usepackage[preprint]{neurips_2019}

\usepackage[utf8]{inputenc} 
\usepackage[T1]{fontenc}    
\usepackage{hyperref}       
\usepackage{url}            
\usepackage{booktabs}       
\usepackage{amsfonts}       
\usepackage{nicefrac}       
\usepackage{microtype}      
\newcommand{\Var}[1]{\mathrm{Var}\left \{ #1\right\}}   
\graphicspath{
	{./images/}{./}{./img/},
}

\title{Linear interpolation gives better gradients than Gaussian smoothing in derivative-free optimization}

\author{%
Albert S Berahas \\
Lehigh University\\
Bethlehem, PA\\
\url{albertberahas@gmail.com}
\And
Liyuan Cao \\
Lehigh University\\
Bethlehem, PA\\
\url{liyuan@lehigh.edu}
\AND
Krzysztof Choromanski\\
Google Brain\\
New York, NY\\
\url{kchoro@google.com}
\And
Katya Scheinberg \\
Lehigh University\\
Bethlehem, PA\\
\url{katyas@lehigh.edu}
}

\begin{document}

\maketitle

\begin{abstract}
\input{abstract}
\end{abstract}

\section{Introduction}
\label{sec:intro}
\input{intro}


\section{Convergence Analysis}
\label{sec:conv_analysis}
\input{converegnce}

\section{Gradient Approximations}
\label{sec:grad_approx}
\input{grad_approx}

\section{Numerical Experiments}
\label{sec:num_exp}
\input{num_exp}

\section{Final Remarks}
\label{sec:fin_rem}
\input{final_remarks}

\newpage
\subsubsection*{Acknowledgements} 
This work was partially supported by the following grants: 2018 Google Faculty Research Award, DARPA Lagrange award HR-001117S0039, NSF TRIPODS 17-40796 grant and NSF CCF 16-18717 grant.


\bibliographystyle{plain}
\bibliography{Katya}



\end{document}

%% file: abstract.tex

In this paper, we consider derivative free optimization problems, where the objective function is smooth but is computed with some amount of noise, the function evaluations are expensive and no derivative information is available. We are motivated by policy optimization problems in reinforcement learning that have recently become popular \cite{ES, choro, choro2, fazel2018global}, and that can be formulated as derivative free optimization problems with the aforementioned characteristics. In each of these works some approximation of the gradient is constructed and a (stochastic) gradient method is applied. In \cite{ES} the gradient information is aggregated along Gaussian directions, while in \cite{choro, choro2} it is computed along orthogonal direction. We provide a convergence rate analysis for a first-order line search method, similar to the ones used in the literature, and derive the conditions on the gradient approximations that ensure this convergence. We then demonstrate via rigorous analysis of the variance and by numerical comparisons on reinforcement learning tasks that the Gaussian sampling method used in \cite{ES} is significantly inferior to the orthogonal sampling used in \cite{choro,choro2} as well as more general interpolation methods. 

%% file: intro.tex
We consider an unconstrained optimization problem of the form
\begin{align}		\label{eq.prob}
	\min_{x \in \mathbb{R}^n} f(x) = \phi(x) + \epsilon(x),
\end{align}
where $f: \mathbb{R}^n\rightarrow \mathbb{R}$ is the output of some black-box procedure, e.g.,  a simulation,  which may be nonsmooth or discontinuous, but is a noisy measurement of a smooth function $\phi$. In this setting, for any given $x\in \mathbb{R}^n$, one is able to obtain (at some cost)  $f(x)$, but  one cannot obtain explicit estimates of $\nabla \phi(x)$. We call this the Derivative-Free Optimization (DFO) setting \cite{ConnScheVice08c,LarsMeniWild2019}. 

We are motivated by the recent increase of  interest in applying and analyzing  DFO methods for policy optimization in reinforcement learning (RL) \cite{ES,choro, choro2, rowland, fazel2018global} as a particular case of simulation optimization. The key step in the majority of these methods is the computation of an estimate of the gradient of the objective function. Since  $\nabla f(x)$ may not exist, we are interested in computing estimates of $\nabla \phi(x)$,  which we denote as $g(x)$ throughout. In the context of RL, $\phi(x)$ may be a smoothing of the noisy reward function. We assume that the noise $\epsilon(x)$ is bounded in absolute value by some constant $\epsilon_f$, but we do not make any other assumption; for example, we do not assume that the noise is stochastic or that it vanishes as we approach the solution. This reflects particular robotics applications for which DFO methods have been shown to be successful \cite{ES}.  
 
In \cite{ES} it was shown that when the objective function evaluations (rollouts in RL) can be performed in parallel, effective gradient estimates can be computed. 
In particular the authors of \cite{ES} used a technique arising from Gaussian smoothing (see e.g., \cite{nesterov2017random, maggiar2018derivative}), which they referred to as {\em evolutionary strategies}. We will not use this terminology in this paper, since it encompasses a large and different  class of optimization methods than those used in \cite{ES} and those considered here. The essence of the Gaussian smoothing method is to compute gradient estimates as a sum of estimates of directional derivatives along Gaussian directions. In particular, $g(x)$ is computed as follows
 \begin{align}		\label{eq:GSG_intro}
	g(x) = \frac{1}{N} \sum_{i=1}^N \frac{f(x+\sigma u_i) - f(x)}{\sigma} u_i, 
\end{align}
 or using the symmetric version
\begin{align}		\label{eq:cGSG_intro}
	g(x) = \frac{1}{2N} \sum_{i=1}^N \frac{f(x+\sigma u_i) - f(x-\sigma u_i)}{\sigma} u_i,
\end{align}
where $\{u_i: i=1, \ldots, N\}$ is a set of random directions that follow a standard Gaussian distribution and $\sigma$ is the sampling radius. 

In follow-up works \cite{choro,choro2, rowland} is was  empirically shown that better gradient estimates can be obtained by using orthogonal directions instead of the Gaussian directions, namely, where the $u_i$'s are chosen to be mutually orthogonal. In the simplest case, this method reduces to the well know finite difference gradient approximation, where $u_i=e_i$, the $i$-th column of the identity matrix. However, it has been observed that for RL tasks randomly chosen sets of orthogonal directions are more effective in practice. We discuss this in more detail in the computational results section. 

While results in \cite{choro, choro2, rowland}  provide empirical confirmation of the advantages of  structured (orthogonal) directions  for various RL benchmark sets, they only scratch the surface of the theory of structured sampling in blackbox optimization. First of all, they do not quantify how gradient accuracy gains depend on the parameters of the training algorithm. This information is often crucial for practitioners.
Moreoever, the key area that is underexplored in \cite{choro, choro2, rowland} is the connection between structured directions and downstream optimization gains. None of these recent papers presents any convergence result for the proposed algorithms. 

Here, we provide a rigorous detailed quantitative analysis that indicates that, when using the same number of samples, methods employing orthogonal directions produce significantly better estimates of the gradient (smaller error) than those employing Gaussian directions. 

 We should note that as an alternative to Gaussian directions, random directions on a unit sphere can also be used to estimate the gradient; see e.g.,  \cite{fazel2018global,flaxman2005online,berahas2019theoretical}. While this method has several theoretical advantages over using Gaussian directions, similar analysis to the one presented in this paper reveals that it is inferior to using orthogonal directions.

\paragraph{Contribitions} The results of this paper can be summarized as follows:
\begin{itemize}[noitemsep,nolistsep,topsep=0pt,leftmargin=20pt]
\item We describe a generic line-search algorithm  adapted to the case of noisy function evaluations with bounded noise for solving \eqref{eq.prob}. 
\item  We establish complexity bounds for this algorithm, in terms of convergence to a neighborhood of an optimal solution defined by the noise, when applied to the minimization of convex, strongly convex and nonconvex functions, under the condition that the gradient estimate $g(x)$ satisfies $\| g(x) - \nabla \phi(x) \| \leq  \theta\|\nabla \phi(x)\|$ for some $\theta \in (0,\sfrac{1}{2}]$. 
\item We then show that if $g(x)$ is computed via linear interpolation of function values using $n$ linearly independent directions $u_i$, and a suitably chosen $\sigma$, i.e., $f(x+\sigma u_i)$,  the above bound is satisfied deterministically. Moreover, we show that $g(x)$ computed using a scaled version of formula \eqref{eq:GSG_intro} using orthonormal directions is equivalent to linear interpolation. 
\item Finally, we analyze the variance of $g(x)$ computed via  \eqref{eq:GSG_intro} using Gaussian directions and show that to satisfy the  bound $\| g(x) - \nabla \phi(x) \| \leq  \theta\|\nabla \phi(x)\|$ with probability $1-\delta$ the number of samples $N$ needs to be greater than $\sfrac{2n}{\theta^2\delta}$, which is significantly greater than $n$. 
\item We support our theoretical bounds and findings with computational experiments on artificial and real problems that arise in reinforcement learning.
\end{itemize}

\paragraph{Organization} The paper is organized as follows. In Section \ref{sec:conv_analysis} we present the analysis of a general gradient descent method with a line search that uses gradient approximations in lieu of the true gradient. We introduce and analyze several methods for approximating the gradient using only function values in Section \ref{sec:grad_approx}. We present a numerical comparison of the gradient approximations and illustrate the performance of different DFO algorithms that employ these gradient approximations in \ref{sec:num_exp}. Finally, in Section \ref{sec:fin_rem}, we make some concluding remarks.

%% file: converegnce.tex
In this section, we analyze a general gradient method with a modified back-tracking line search in the DFO setting. The results presented in this section are an adaptation of those presented in \cite{berahas2019derivative} and  \cite{berahas2019theoretical}. 

We consider an iteration of the form:
\begin{align}	\label{eq.iteration}
	x_{k+1} = x_k - \alpha_k g(x_k),
\end{align}
where $g(x_k)$ is an approximation to the gradient constructed using only evaluations of $f$, and $\alpha_k$, the step size parameter, is chosen to satisfy the relaxed Armijo condition
\begin{align}	\label{eq.rel_Arm}
	f(x_k - \alpha_k g(x_k)) \leq f(x_k) - c_1 \alpha_k \| g(x_k) \|^2 + 2\epsilon_f,
\end{align}
where $c_1 \in (0,1)$, for some $\epsilon_f \geq 0$. If a trial value $\alpha_k$ does not satisfy \eqref{eq.rel_Arm}, the step size parameter is set to a fixed fraction $\tau<1$ of the previous value, i.e., $\alpha_k = \tau \alpha_k$.

We make the following assumptions.
\begin{assum} \label{assum.func_bnd} \textbf{(Boundedness of Noise in the Function)} There is a constant $\epsilon_f \geq 0$ such that $ | f(x) - \phi(x)| = |e(x)| \leq \epsilon_f$ for all $x \in \mathbb{R}^n$.
\end{assum}

\begin{assum} \label{assum.lip_grad} \textbf{(Lipschitz Continuity of the Gradients of $\pmb{\phi}$)} The function $\phi$ is continuously differentiable, and the gradients of $\phi$ are $L$-Lipschitz continuous for all $x \in \mathbb{R}^n$.
\end{assum}

We establish results under the \emph{norm condition} \cite{carter1991global} given by
\begin{align}		\label{eq.normcond}
	\| g(x) - \nabla \phi(x) \| \leq \theta \| \nabla \phi (x) \|,
\end{align}
for $\theta \in [0,\sfrac{1}{2})$, which implies
\begin{align}
	(1-\theta) \| \nabla \phi (x) \| \leq \| g(x) \| \leq (1+\theta) \| \nabla \phi(x) \|.
\end{align}

\begin{lem}	\label{lem.1} Let Assumptions \ref{assum.func_bnd} and \ref{assum.lip_grad} hold. If 
\begin{align}		\label{eq.alpha_bar}
	\alpha_k \leq \bar{\alpha} = \frac{2(1-2\theta - c_1(1-\theta)}{L(1-\theta)}.
\end{align}
for $k=0,1,\dots$, and \eqref{eq.normcond} holds, then the relaxed Armijo condition \eqref{eq.rel_Arm} is satisfied. Moreover, 
\begin{align}		\label{eq.decrease_phi}
	\phi(x_{k+1}) \leq \phi(x_k) - c_1 \tau \bar{\alpha} (1-\theta)^2 \| \nabla \phi(x_k)\|^2 + 4 \epsilon_f.
\end{align}
\end{lem}

\begin{proof} Since $\phi$ satisfies \ref{assum.lip_grad} and \eqref{eq.normcond} holds, we have
\begin{align*}
	\phi(x_k - \alpha_k g(x_k)) & \leq \phi(x_k) - \alpha_k g(x_k)^T \nabla \phi (x_k) + \frac{\alpha_k^2 L}{2} \| g(x_k)\|^2\\
			& = \phi(x_k) - \alpha_k g(x_k)^T (\nabla \phi (x_k) - g(x_k)) - \alpha_k \left[1 -\frac{\alpha_k L}{2}\right] \| g(x_k)\|^2\\
			& \leq \phi(x_k) + \alpha_k \| g(x_k)\| \|\nabla \phi (x_k) - g(x_k)\| - \alpha_k \left[1 -\frac{\alpha_k L}{2}\right] \| g(x_k)\|^2\\
			& \leq \phi(x_k) + \frac{\alpha_k \theta}{1 - \theta }\| g(x_k)\|^2 - \alpha_k \left[1 -\frac{\alpha_k L}{2}\right] \| g(x_k)\|^2\\
			& = \phi(x_k) - \alpha_k \left[\frac{1-2\theta}{1-\theta} -\frac{\alpha_k L}{2}\right] \| g(x_k)\|^2.
\end{align*}
By Assumption \ref{assum.func_bnd}, we have
\begin{align*}
	f(x_k - \alpha_k g(x_k)) \leq f(x_k) - \alpha_k \left[\frac{1-2\theta}{1-\theta} -\frac{\alpha_k L}{2}\right] \| g(x_k)\|^2 + 2 \epsilon_f.
\end{align*}
From this we conclude that \eqref{eq.rel_Arm} holds whenever 
\begin{align*}
	f(x_k) - \alpha_k\left[\frac{1-2\theta}{1-\theta} -\frac{\alpha_k L}{2}\right] \| g(x_k)\|^2 + 2 \epsilon_f \leq f(x_k) - c_1 \alpha_k \| g(x_k) \|^2 + 2\epsilon_f,
\end{align*}
which is equivalent to \eqref{eq.alpha_bar}.

We have shown that when $\alpha_k \leq \bar{\alpha}$ the relaxed Armijo condition is \eqref{eq.rel_Arm} is satisfied. Since we find $\alpha_k$ using a constant backtracking factor of $\tau <1$, we have that $\alpha_k > \tau \bar{\alpha}$. Therefore, using Assumption \eqref{assum.func_bnd}, we have
\begin{align*}
	\phi(x_{k+1}) &\leq \phi(x_k) - c_1 \alpha_k \| g(x_k) \|^2 + 4\epsilon_f \\
				& \leq \phi(x_k) - c_1 \tau \bar{\alpha} (1-\theta)^2 \| \nabla \phi(x_k) \|^2 + 4\epsilon_f,
\end{align*}
which completes the proof.
\end{proof}

\subsection{Convex Functions}

In this section, we state and prove results for the case where the function $\phi$ is  convex. We make the following additional standard assumption.
\begin{assum} \label{assum.conv} \textbf{(Convexity and bounded level sets of $\pmb{\phi}$)}  The function $\phi$ is convex and has bounded level sets, i.e., 
\begin{align*}
	\|x - x^\star\| \leq D, \quad \text{for all $x$ with } \ f(x) \leq f(x_0),	
\end{align*}
where $x^\star$ is a global minimizer of $\phi$. Let $\phi^\star = \phi(x^\star)$.
\end{assum}

\begin{theorem} \label{thm.conv} Suppose that Assumptions \ref{assum.func_bnd}, \ref{assum.lip_grad} and \ref{assum.conv} hold. Let $\{ x_k \}$ be the iterates  generated by \eqref{eq.iteration}, where  $\alpha_k$ satisfies the relaxed Armijo condition \eqref{eq.rel_Arm}. Then, for $k=0,1,\dots$, 
\begin{align}	\label{eq.convex}
	\phi(x_{k}) - \phi^\star \leq \max \left\{ \frac{D^2}{k(1-\gamma)\eta}, \frac{2D \sqrt{\epsilon_f}}{\sqrt{\gamma \eta}}+4\epsilon_f\right\}, 
\end{align}
where $\eta = c_1 \tau \bar{\alpha}(1-\theta)^2$, and $\bar{\alpha}$ is given in \eqref{eq.alpha_bar}.
\end{theorem}

\begin{proof} By Assumption \ref{assum.conv}, we have
\begin{align}	\label{eq.convex_D}
	\phi(x_k) - \phi^\star \leq \nabla \phi(x_k)^T (x_k - x^\star) \leq \| \nabla \phi(x_k)\| \| x_k - x^\star \| \leq D \| \nabla \phi(x_k)\|.
\end{align}
Let $z_k = \phi(x_k) - \phi^\star$; by \eqref{eq.decrease_phi} and \eqref{eq.convex_D} we have,
\begin{align*}
	z_{k} - z_{k+1} \geq \frac{c_1 \tau \bar{\alpha} (1-\theta)^2 z_k^2}{D^2} - 4\epsilon_f.
\end{align*}
We thus have,
\begin{align}	\label{eq.convex1}
	\frac{1}{z_{k+1}} - 	\frac{1}{z_{k}} = \frac{z_k - z_{k+1}}{z_{k+1} z_k } \geq \frac{z_k - z_{k+1}}{ z_k^2 } \geq \frac{c_1 \tau \bar{\alpha} (1-\theta)^2 }{D^2} - \frac{4\epsilon_f}{z_k^2}.
\end{align}
We now consider two regimes: $(i)$ $z_i> \frac{2D \sqrt{\epsilon_f}}{\sqrt{\gamma c_1 \tau \bar{\alpha}}(1-\theta)}$, $\forall 0\leq i\leq k-1$ and $(ii)$ $z_i \leq \frac{2D \sqrt{\epsilon_f}}{\sqrt{\gamma c_1 \tau \bar{\alpha}}(1-\theta)}$, for some $0\leq i\leq k-1$,  where $\gamma \in (0,1)$. In the former case, the optimality gap is large compared to the noise, and thus by \eqref{eq.convex1} we have  $\forall 0\leq i\leq k$
\begin{align}	\label{eq.convex2}
	\frac{1}{z_{i+1}} - 	\frac{1}{z_{i}} \geq (1-\gamma) \frac{c_1 \tau \bar{\alpha} (1-\theta)^2 }{D^2}.
\end{align}
By aggregating this bounds for all  $\forall 0\leq i\leq k-1$ formula recursively, we obtain
\begin{align}	\label{eq.convex3}
	\frac{1}{z_{k}} \geq \frac{1}{z_{0}} + k(1-\gamma) \frac{c_1 \tau \bar{\alpha} (1-\theta)^2 }{D^2} \geq k(1-\gamma) \frac{c_1 \tau \bar{\alpha} (1-\theta)^2 }{D^2},
\end{align}
which yields the first part of the result in \eqref{eq.convex}. The second part of the result is obtained using three facts: the fact that there exists $i<k$ such 
that $z_i\leq \frac{2D \sqrt{\epsilon_f}}{\sqrt{\gamma c_1 \tau \bar{\alpha}}(1-\theta)}$, the fact that for all $i$ $z_{i+1}-z_i\leq 4 \epsilon_f$ and for all
$i$ such that $z_i> \frac{2D \sqrt{\epsilon_f}}{\sqrt{\gamma c_1 \tau \bar{\alpha}}(1-\theta)}$, $z_{i+1}-z_i\leq 0$. 
\end{proof}

 \begin{remark} The value $\phi^\star+\frac{2D \sqrt{\epsilon_f}}{\sqrt{\gamma \eta}}+4\epsilon_f$ can be interpreted as the lowest function value that is guaranteed to be achived in the presence of noise. 
\end{remark}

\subsection{Strongly Convex Functions}

In this section, we state and prove results for the case where the function $\phi$ is strongly convex.
\begin{assum} \label{assum.str_conv} \textbf{(Strong Convexity of $\pmb{\phi}$)}  There exist a constant $\mu>0$ such that, for all $x,y \in \mathbb{R}^n$,
\begin{align*}
	\phi(y) \geq \phi(x) + \nabla \phi(x)^T (y-x) + \frac{\mu}{2} \| x - y \|^2.
\end{align*}
Under Assumption \ref{assum.str_conv}, let $\phi^\star = \phi(x^\star)$, where $x^\star$ is the minimizer of $\phi$.
\end{assum}

\begin{theorem} \label{thm.str} Suppose that Assumptions \ref{assum.func_bnd}, \ref{assum.lip_grad} and \ref{assum.str_conv} hold. Let $\{ x_k \}$ be the iterates  generated by \eqref{eq.iteration}, where  $\alpha_k$ satisfies the relaxed Armijo condition \eqref{eq.rel_Arm}. Then, for $k=0,1,\dots$, 
\begin{align}	\label{eq.strong}
	\phi(x_{k}) - \left[\phi^\star + \frac{4\epsilon_f}{1-\rho}\right]\leq \rho^k \left(\phi(x_0) -\left[\phi^\star + \frac{4\epsilon_f}{1-\rho}\right]\right), 
\end{align}
where $\rho = 1 - 2 \mu c_1 \tau \bar{\alpha}(1-\theta)^2$, and $\bar{\alpha}$ is given in \eqref{eq.alpha_bar}.
\end{theorem}

\begin{proof} Starting with \eqref{eq.decrease_phi}, by strong convexity we have $\| \nabla \phi (x_k)\|^2 \geq 2 \mu (\phi(x_k) - \phi^\star)$, thus
\begin{align*}
	\phi(x_{k+1}) - \phi^\star &\leq \phi(x_k) - \phi^\star - 2\mu c_1 \tau \bar{\alpha} (1-\theta)^2 (\phi(x_k) - \phi^\star) + 4 \epsilon_f\\
		& = \rho(\phi(x_k) - \phi^\star) + 4 \epsilon_f.
\end{align*}
Subtracting $\frac{4 \epsilon_f}{1 - \rho}$ from both sides, 
\begin{align*}
	\phi(x_{k+1}) - \phi^\star - \frac{4 \epsilon_f}{1 - \rho} &\leq  \rho(\phi(x_k) - \phi^\star) + 4 \epsilon_f - \frac{4 \epsilon_f}{1 - \rho}\\
			& = \rho(\phi(x_k) - \phi^\star) + \frac{(1- \rho)4 \epsilon_f - 4 \epsilon_f}{1 - \rho}\\
			& = \rho(\phi(x_k) - \phi^\star) - \frac{\rho4 \epsilon_f}{1 - \rho}\\
			& = \rho \left(\phi(x_k) - \phi^\star - \frac{4 \epsilon_f}{1 - \rho} \right).
\end{align*}
Recursive application of the above yields the desired result.
\end{proof}

\begin{remark} We interpret the term $\left[\phi^\star + \frac{4\epsilon_f}{1-\rho}\right]$ in \eqref{eq.strong} as the lowest value of the objective that is guaranteed to be achieved in the presence of noise. 
\end{remark}

\subsection{Nonconvex Functions}

In this section, we state and prove results for the case where the function $\phi$ is nonconvex.

\begin{assum} \label{assum.lower_bnd} \textbf{(Lower Bound on $\pmb{\phi}$)} The function $\phi$ is bounded below by a scalar $\hat{\phi}$.
\end{assum}

\begin{theorem} \label{thm.nonconvex} Suppose that Assumptions \ref{assum.func_bnd}, \ref{assum.lip_grad} and \ref{assum.lower_bnd} hold. Let $\{ x_k \}$ be the iterates  generated by \eqref{eq.iteration}, where $\alpha_k$ satisfies the relaxed Armijo condition \eqref{eq.rel_Arm}. Then, for any $T\geq 1$,
\begin{align*}
	\frac{1}{T} \sum_{k=0}^{T-1} \| \nabla \phi(x_k)\|^2 &\leq \frac{\phi(x_0) - \hat{\phi}}{\eta T} + \frac{4 \epsilon_f}{\eta} \xrightarrow[]{T\rightarrow \infty} \frac{4 \epsilon_f}{\eta}.
\end{align*}
where $\eta =  c_1 \tau \bar{\alpha}(1-\theta)^2$, and $\bar{\alpha}$ is given in \eqref{eq.alpha_bar}.
\end{theorem}

\begin{proof} By \eqref{eq.decrease_phi}, we have
\begin{align*}
	\phi(x_{k+1}) \leq \phi(x_k) - c_1 \tau \bar{\alpha} (1-\theta)^2 \| \phi(x_k)\|^2 + 4 \epsilon_f,
\end{align*}
and thus
\begin{align*}
	\| \phi(x_k)\|^2 \leq \frac{\phi(x_k) - \phi(x_{k+1})}{c_1 \tau \bar{\alpha}(1-\theta)^2} + \frac{4 \epsilon_f}{c_1 \tau \bar{\alpha}(1-\theta)^2}
\end{align*}
Summing over the first $T-1$ iterations, 
\begin{align*}
	\sum_{k=0}^{T-1} \| \phi(x_k)\|^2 &\leq \sum_{k=0}^{T-1} \frac{\phi(x_k) - \phi(x_{k+1})}{c_1 \tau \bar{\alpha}(1-\theta)^2} + \sum_{k=0}^{T-1}\frac{4 \epsilon_f}{c_1 \tau \bar{\alpha}(1-\theta)^2}\\
			& \leq \frac{\phi(x_0) - \hat{\phi}}{c_1 \tau \bar{\alpha}(1-\theta)^2} + \frac{4 \epsilon_fT}{c_1 \tau \bar{\alpha}(1-\theta)^2},
\end{align*}
where $\phi(x_0) - \hat{\phi} \geq \phi(x_0) - \phi(x_T)$. Averaging over the first $T-1$ iterations yields the desired result.
\end{proof}

 \begin{remark} 
 $\sqrt{\frac{4 \epsilon_f}{\eta}}$ can interpreted as the lowest value of the norm of the  gradient that can be achieved in the presence of noise. 
 \end{remark}

\subsection{General Remarks}

In summary, the results presented in this section for a modified line search algorithm with noise show that the standard convergence rates hold until certain accuracy related to $\epsilon_f$  has been reached.  These convergence results only require the norm condition \eqref{eq.normcond} on the gradient estimate, but in the next section, we establish the norm condition under specific relation between $\|\nabla \phi(x)\|$ and $\epsilon_f$.

%% file: grad_approx.tex
In this section, we compare two approaches for computing gradient approximations $g$ using only samples of function values $f(x)$. The first method is based on linear interpolation of these function values and the second is based on aggregated estimates of directional derivatives along  Gaussian directions.

\subsection{Linear Interpolation Models}	\label{sec:lin_mod}
Interpolation models have a long history in the DFO setting; see e.g., \cite{ConnScheToin97,ConnScheVice08c,ConnScheToin98,Powe74,Powe06,wild2008orbit,maggiar2018derivative}. These methods construct surrogate models of the objective function using interpolation or regression. While these methods usually construct quadratic models around $x \in \mathbb{R}^n$ 
we focus on the simplest case of linear models, 
\begin{align}		\label{eq:lin_mod}
	m(y)=f(x)+g(x)^T (y-x),
\end{align}
since for our large-scale applications we assume that one cannot compute the  number of  function evaluations required to construct quadratic models. 



Let us consider the following sample set  $\mathcal{X}=\{x+\sigma u_1, x+\sigma u_2, \ldots, x+\sigma u_n\}$ for some $\sigma >0$. In other words, we have $n$ directions denoted by $u_i$ and sample $f$ along those directions, around $x$, using step size $\sigma$. We also assume we know $f(x)$.  Let $F_{\cal X} \in \mathbb{R}^n$ be a vector whose entries are
$f(x+\sigma u_i) - f(x)$ for $i=1\dots n$, and let $Q_{\cal X} \in \mathbb{R}^{n \times n}$ define a matrix whose rows are given by $u_i$ for $i=1\dots n$. 
Model $m(u)=f(x)+g(x)^T (u-x)$ is constructed to satisfy  the interpolation conditions $f(x+\sigma u_i)=m(x+\sigma u_i)$ for all $i=1, \ldots, n$
which  can be written as
\begin{align}\label{eq:main_system}
	\sigma Q_{\cal X} g(x)=F_{\cal X}. 
\end{align}

If the matrix $Q_{\cal X}$ is nonsingular, then  $g(x) = \frac{1}{\sigma}Q_{\cal X}^{-1}F_{\cal X}$. When $Q_{\cal X}$ is the identity matrix, then we recover standard forward finite difference gradient estimation. In the specific case when $Q_{\cal X}$ is orthonormal, then 
$Q_{\cal X}^{-1}=Q_{\cal X}^T$, thus $g(x)$ is written as 
\begin{align}\label{eq:orthogonal_sampling}
g(x)=\sum_{i=1}^n \frac{f(x+\sigma u_i)-f(x)}{\sigma}u_i, 
\end{align}
which is a scaled version of  \eqref{eq:GSG_intro} with orthonormal $u_i$'s.  The difference in scaling between \eqref{eq:orthogonal_sampling} and \eqref{eq:GSG_intro} 
comes from the fact the difference in expected length of $u_i$ when $u_i$'s are orthonormal and when they are Gaussian. 

We show a bound on $\|g-\nabla \phi(x)\|$, which is an extension of results in
 \cite{ConnScheVice08c,conn2008geometry} to include the error term. 

\begin{theorem} \label{thm:bnd_linmod} Let $\mathcal{X}=\{x,x+\sigma u_1,\dots,x+\sigma u_n \}$ be set of interpolation points such that $\max_i\|u_i\|\leq 1$  and that 
$Q_{\cal X}$ is nonsingular. Assume that $f(x)=\phi(x)+\epsilon(x)$, where $\phi(x)$ satisfies Assumption \ref{assum.lip_grad} and $\epsilon(x)\leq \epsilon_f$ for all $x$. Then, 
\begin{align}
	\|g({x}) - \nabla \phi(x)\| \leq \frac{\| Q_{\cal X}^{-1}\|_2\sqrt{n}\sigma L}{2} +\frac{2\| Q_{\cal X}^{-1}\|_2\sqrt{n}\epsilon_f}{\sigma}.
\end{align}
\end{theorem}
\begin{proof}
From the interpolation conditions and the mean value theorem $\forall i=1, \ldots, n$ we have 
\begin{align}
\sigma g(x)^T u_i &=f(x+\sigma u_i)-f(x)=\phi(x+\sigma u_i)-\phi(x)+\epsilon(x+\sigma u_i)-\epsilon(x)\\
&=\int_{0}^t\sigma u_i^T\nabla \phi(x+t\sigma u_i)dt+\epsilon(x+\sigma u_i)-\epsilon(x).
\end{align}
From the $L$-smoothness of $\phi(\cdot)$ and the bound on $\epsilon(\cdot)$ we have 
\begin{align}
\sigma |(g(x)-\nabla \phi(x))^Tu_i| \leq  \frac {L\sigma^2 \|u_i\|^2}{2}+2\epsilon_f, \quad \forall i=1, \ldots, n
\end{align}
which in turn implies 
\begin{align}
 \| Q_{\cal X}(g(x)-\nabla \phi(x))\| \leq  \frac {L\sigma \sqrt{n} }{2}+\frac{2 \sqrt{n} \epsilon_f}{\sigma},
\end{align}
and the theorem statement follows. 
\end{proof}

It is clear that the best bound on $\|g({x}) - \nabla \phi({x})\| $ is obtained when $ Q_{\cal X}$ is orthonormal, and henceforth we assume that this is the case. 

We now establish conditions on $\sigma$ and  $\| \nabla \phi({x})\|$ for which $\|g({x}) - \nabla \phi({x})\| \leq \theta \|\nabla \phi({x})\|$, for some given 
$\theta\in [0,1)$ and given bounds on the noise $\epsilon_f$. From Theorem \ref{thm:bnd_linmod} we need
\begin{align}
\frac {\sqrt{n}\sigma L}{2}+\frac{2 \sqrt{n} \epsilon_f}{\sigma}\leq \theta \|\nabla \phi({x})\|, \ \text{ or }\ \frac {\sqrt{n} \sigma^2 L }{2}- \theta \|\nabla \phi({x})\|{\sigma} + 2 \sqrt{n} \epsilon_f\leq 0. 
\end{align}
This is achieved by any $\sigma$ in the range 
\begin{align}
\frac{\theta \|\nabla \phi(x)\| -\sqrt{\theta^2 \|\nabla \phi(x)\|^2- 4Ln\epsilon_f}}{\sqrt{n}L}\leq \sigma \leq 
\frac{\theta \|\nabla \phi(x)\| +\sqrt{\theta^2 \|\nabla \phi({x})\|^2- 4Ln\epsilon_f}}{\sqrt{n}L},
\end{align}
as long as 
\vspace{-2pt}
\begin{align}\label{eq:bnd_grad_epsf}
 \|\nabla \phi(x)\| \geq \frac{ 2\sqrt{Ln\epsilon_f}}{\theta}. 
\end{align}
In Section \ref{sec:conv_analysis} we showed that when $\|g({x}) - \nabla \phi({x})\| \leq \theta \|\nabla \phi({x})\|$ holds for $\theta<1/2$, then the related line search algorithm 
converges to a neighborhood of the solution with essentially the same rate as a gradient-based line search or descent method. The neighborhood is defined by $\epsilon_f$ and the smaller this error is, the closer to the solution the algorithm converges. Here we note that the smaller the $\epsilon_f$, the easier it is to satisfy $\|g({x}) - \nabla \phi({x})\| \leq \theta \|\nabla \phi({x})\|$ by choosing appropriate $\sigma$, when \eqref{eq:bnd_grad_epsf} holds. In particular for $\epsilon_f=0$ any $\sigma$ less than or equal to $\theta \|\nabla \phi(x)\| /(\sqrt{n}L)$ works.

\begin{remark} A bound similar to the one in Theorem  \ref{thm:bnd_linmod}  can be derived for the symmetric formula \eqref{eq:cGSG_intro} when $u_i$'s are orthonormal.
Under assumption that $\nabla^2 \phi(x)$ is Lipschitz continuous, the first term in the bound decays as $\sigma^2$, instead of $\sigma$, however, in the presence of noise, this produces limited benefit and puts tighter restrictions on $\sigma$, hence we do not consider this version in this paper. 
\end{remark}

\subsection{Estimates of the Gradients of  Gaussian Smoothing}

 Gaussian smoothing has recently become a popular tool for building gradient approximations using only function values. This approach has been exploited in several recent papers; see e.g., \cite{nesterov2017random, maggiar2018derivative,ES, NES2}. Gaussian smoothing of a given function $f$ is obtained as  follows:
\begin{align}	\label{eq: ES obj}
\phi(x) 
&= \mathbb{E}_{u \sim \mathcal{N}({0}, I)} [f(x+\sigma u)] = \int_{\mathbb{R}^n} f(x+\sigma u) \pi (u|0,I) du,  
\end{align}
where 
$\mathcal{N}({0}, I)$ denotes the standard multivariate normal distribution. The function $\pi(u|0,I)$ is the probability density function (pdf) of $\mathcal{N}(0,I)$ evaluated at $u$.  

The gradient of $\phi$ can be expressed as
\begin{align}		\label{eq: ES g}
    &\nabla  \phi(x) = \frac{1}{\sigma} {\mathbb{E}}_{u \sim \mathcal{N}({0}, I)} [f(x+\sigma u) u]. 
\end{align}
The approach used in \cite{ES} is to approximate $\nabla \phi(x)$ by sample average approximations applied to \eqref{eq: ES g}. Here we will show that this approach produces less accurate gradients than  the methods described in the previous section for the same number of samples. In \cite{nesterov2017random,maggiar2018derivative} Gaussian smoothing is applied to functions that are possibly nonsmooth but Lipschitz continuous. In this section we will thus impose the same assumption.
\begin{assum} \label{assum.lip_cont} \textbf{(Lipschitz Continuity of $\pmb{f}$)} The function $f$ is $L_f$-Lipschitz continuous  for all $x \in \mathbb{R}^n$.
\end{assum}
It 
has been shown in \cite{nesterov2017random} that under Assumption \ref{assum.lip_cont}  $|\phi(x)-f(x)|\leq \epsilon_f=\sigma\sqrt{n} L_f$ and
$\phi(x)$ is Lipschitz smooth with constant $L=\sqrt{n} L_f/\sigma$.  Note that it is possible to derive 
similar bound on functions $f$ that are discontinuous with particular assumptions on the discontinuity, however, as we will show, even under Assumption \ref{assum.lip_cont} computing gradient estimates by sample averaging is more costly than by the interpolation methods.

%

Applying sample average approximations  to \eqref{eq: ES g}, yields
\begin{align}		\label{eq:ncGSG}
	g(x)=\frac{1}{N\sigma} \sum_{i=1}^N f(x+\sigma u_i) u_i, 
\end{align}
where $u_i \sim \mathcal{N}({0}, I)$ for $i = 1,2,\dots, N$. It can be easily shown that $g(x)$ computed via \eqref{eq:ncGSG} 
has large variance (the variance explodes as $\sigma$ goes to $0$). The following simple modification, 
\begin{align}		\label{eq:GSG}
	g(x) = \frac{1}{N} \sum_{i=1}^N \frac{f(x+\sigma u_i) - f(x)}{\sigma} u_i, 
\end{align}
 eliminates this problem  and is indeed  used in practice instead of \eqref{eq:ncGSG}. Note that  the expectation of \eqref{eq:GSG} is also  $\nabla \phi(x)$, since ${\mathbb{E}}_{u \sim \mathcal{N}(\mathbf{0}, I)} f(x) u=0$. 
In what follows we will refer to $g(x)$ computed via \eqref{eq:GSG}  as the Gaussian smoothed gradient (GSG). As pointed out in \cite{nesterov2017random}, $\frac{f(x+\sigma u_i) - f(x)}{\sigma} u_i$ can be 
interpreted as a forward finite difference version of the directional derivative of $f$ at $x$ along $u_i$. One can also consider \eqref{eq:cGSG_intro}, which is  the central difference variant of \eqref{eq:GSG}.

The properties of \eqref{eq: ES obj} and \eqref{eq:GSG}, with $N=1$, were analyzed in \cite{nesterov2017random}. However, this analysis does not explore the effect of $N>1$ on the variance of  $g(x)$. On the other hand, in \cite{ES},  GSG estimates  are computed using both \eqref{eq:GSG} and \eqref{eq:cGSG_intro}  with large samples sizes $N$
in a fixed step size gradient descent algorithm, but without any analysis or discussion of the choices of $N$, $\sigma$ or $\alpha$. Thus, the purpose of this section is to derive bounds on the approximation error $\| g(x) - \nabla \phi(x) \|$ for GSG, and to derive conditions of $\sigma$ and $N$ under which condition \eqref{eq.normcond} holds, and thus so do the convergence results for the line search DFO algorithm based on these approximations. While \eqref{eq:cGSG_intro}  is  used in practice, it does not yield better gradient approximation. We omit the analysis of \eqref{eq:cGSG_intro} here for brevity, although it is similar to the analysis of \eqref{eq:GSG} which we provide.

The variance for \eqref{eq:GSG} can be expressed as 
\begin{align} \label{eq:var GSG}  
	\Var{g(x)} = \frac{1}{N} \mathbb{E}_{u \sim \mathcal{N}(0,I)} \left[\left(\frac{f(x+\sigma u) - f(x)}{\sigma} \right)^2 u u^T \right] - \frac{1}{N} \nabla \phi(x) \nabla \phi(x)^T . 
\end{align}

First we state some properties of normally distributed multivariate random variable $u \in \mathbb{R}^n$. 
\begin{align}\label{eq:mvn}
	&\mathbb{E}_{u \sim \mathcal{N}({0},I)} \left[ u u^T \right]= I \nonumber\\
		&\mathbb{E}_{u \sim \mathcal{N}({0},I)} \left[ (u^Tu) u u^T \right]= (n+2)I \nonumber\\
	&\mathbb{E}_{u \sim \mathcal{N}({0},I)} \left[(a^T u)^2 u u^T \right]= a^T a I + 2 a a^T \nonumber\\
	&\mathbb{E}_{u \sim \mathcal{N}({0},I)} \left[a^T u \cdot u^T u \cdot u u^T \right]= 0_{n\times n} \\
	&\mathbb{E}_{u \sim \mathcal{N}({0},I)} \left[(u^T u)^2 u u^T \right]= (n+2)(n+4) I \nonumber\\
         &\mathbb{E}_{u \sim \mathcal{N}({0},I)} \left[a^T u \|u\|^3 \right]= 0,\nonumber\\
	&\mathbb{E}_{u \sim \mathcal{N}({0},I)} \left[(u^T u)^3 u u^T \right]= (n+2)(n+4)(n+8) I,\nonumber
\end{align}
where $a$ is any vector in $\mathbb{R}^n$ independent from $u$. We now provide bounds for the variance of GSG.

\begin{lem} \label{lemma:bnd_var}
Under Assumption  \ref{assum.lip_cont}, if $g(x)$ is calculated by \eqref{eq:GSG}, then, for all $x \in \mathbb{R}^n$,
\begin{align*}	
	\Var{g(x)} \preceq \kappa(x) I, \ \ \text{where} \ \ \kappa(x) = \frac{8 \|\nabla \phi(x)\|^2 +  L_f^2 n (n+2)(n+4)+8n(n+2)L_f^2n+16nL_f^2}{4N}.
\end{align*} 

\end{lem}

\begin{proof} 
By \eqref{eq:var GSG}, we have 
\begin{align*}
	\Var{g(x)} &= \frac{1}{N} \mathbb{E}_{u \sim \mathcal{N}({0},I)} \left[\left(\frac{f(x+\sigma u) - f(x)}{\sigma} \right)^2 u u^T \right ]- \frac{1}{N} \nabla \phi(x) \nabla \phi(x)^T \\
	&\preceq \frac{1}{N \sigma^2} \mathbb{E}_{u \sim \mathcal{N}({0},I)} \left[\left( \nabla \phi(x)^T \sigma u + \frac{1}{2}L \sigma^2 u^T u+2\epsilon_f\right)^2 u u^T   \right]- \frac{1}{N} \nabla \phi(x) \nabla \phi(x)^T\\
	&= \frac{1}{N \sigma^2} \mathbb{E}_{u \sim \mathcal{N}({0},I)} \Big[ \sigma^2 (\nabla \phi(x)^T u)^2 u u^T + L \sigma^3 (\nabla \phi(x)^T u) ( u^T u )u u^T  + \frac{1}{4} L^2 \sigma^4 (u^T u)^2 u u^T\\
	&\quad   +4 \epsilon_f\sigma (\nabla \phi(x)^T u)uu^T + 2\epsilon_f L \sigma^2 (u^T u)uu^T +4\epsilon^2_f uu^T \Big] - \frac{1}{N} \nabla \phi(x) \nabla \phi(x)^T \\
	&\stackrel{\mathclap{\mathrm{\eqref{eq:mvn}}}}{=} \frac{1}{N } \left( \nabla \phi(x)^T \nabla \phi(x) I + 2 \nabla \phi(x) \nabla \phi(x)^T\right) - \frac{1}{N} \nabla \phi(x) \nabla \phi(x)^T \\ &
	+ \quad \frac{1}{\sigma^2N }\left( L \sigma^3 \cdot 0 + \frac{1}{4} L^2 \sigma^4 (n+2)(n+4)I +4 \epsilon_f\sigma\cdot 0+2(n+2)\epsilon_fL\sigma^2I+4\epsilon^2_fI\right)  \\
	&\preceq \frac{ 8 \|\nabla \phi(x)\|^2 +   (n+2)(n+4)L^2\sigma^2 +8(n+2)\epsilon_fL}{4N}  I+\frac{4\epsilon^2_f}{\sigma^2N}I,
\end{align*}
where the first inequality comes from the Lipschitz continuity of the gradient of $\phi(x)$ and the bound  $|f(x)-\phi(x)|\leq \epsilon_f$
 and the last inequality is due to $\nabla \phi(x) \nabla \phi(x)^T \preceq \nabla \phi(x)^T \nabla \phi(x) I$. 

Finally, we recall that for our choice of $\phi$, $L=\sqrt{n}L_f/\sigma$ and  $\epsilon_f=\sqrt{n}L_f\sigma$. Substituting these expressions in the above yields the desired result. 
\end{proof}

Using the result of Lemma \ref{lemma:bnd_var}, we can now bound the quantity $\| g(x) - \nabla \phi(x) \|$, in probability, using Chebyshev's inequality. 

\begin{theorem} \label{lem:prob_bnd_smoothed} 
Let $\phi$ be a Gaussian smoothed approximation of $f$ \eqref{eq: ES obj}. Under Assumption  \ref{assum.lip_cont}, if $g(x)$ is  calculated via \eqref{eq:GSG} with sample size
\begin{align*}
	N \ge \frac{2n \|\nabla \phi(x)\|^2}{\delta r^2} +\frac{L_f^2 n (n+2)(n+4)+8n(n+2)L_f^2+16nL_f^2}{4\delta r^2},
\end{align*}
then, for all $x \in \mathbb{R}^n$, $\|g(x) - \nabla \phi(x)\| \le r$ holds with probability at least $1 - \delta$, for any $r>0$ and $0< \delta<1$. 
\end{theorem}

\begin{proof}
By Chebyshev's inequality, for any $r > 0$, we have
\begin{align*}
\mathbb{P}\left\{\sqrt{(g(x) - \nabla \phi(x))^T \Var{g(x)}^{-1} (g(x) - \nabla \phi(x))} > r \right\} 
&\le \frac{n}{r^2}. 
\end{align*}
Since $\Var{g(x)} \preceq \kappa I$, we have $\Var{g(x)}^{-1} \succeq \kappa^{-1} I$ and 
\[ \sqrt{(g(x) - \nabla \phi(x))^T \Var{g(x)}^{-1} (g(x) - \nabla \phi(x))} \ge \kappa^{-\frac{1}{2}} \|g(x) - \nabla \phi(x)\|.
\]
Therefore, we have, 
\begin{align*}
\mathbb{P}\left\{\kappa^{-\frac{1}{2}} \|g(x) - \nabla \phi(x)\| > r \right\} \le \frac{n}{r^2}  \ \ \Longrightarrow \ \
\mathbb{P}\left\{ \|g(x) - \nabla \phi(x)\| > r \right\} \le \frac{\kappa n}{r^2}. 
\end{align*}
By Lemma \ref{lemma:bnd_var}, for GSG we have
\begin{align*}
	\mathbb{P}\left\{ \|g(x) - \nabla \phi(x)\| > r \right\} &\le \frac{2n \|\nabla \phi(x)\|^2}{N r^2} +
	\frac{L_f^2 n (n+2)(n+4)+8n(n+2)L_f^2+16nL_f^2}{4N r^2}.
\end{align*} 
Thus when 
$$
N \ge \frac{2n \|\nabla \phi(x)\|^2}{\delta r^2} + \frac{L_f^2 n (n+2)(n+4)+8n(n+2)L_f^2+16nL_f^2}{4\delta r^2},
$$
 we have $\mathbb{P}\left\{ \|g(x) - \nabla \phi(x)\| > r  \right\} \le \delta$.

\end{proof}

We now derive the bound on $N$  under which the norm condition \eqref{eq.normcond} holds with some given probability $1-\delta$. Thus, we set $r=\theta\|\nabla \phi(x)\|$, and for proper comparison with the case of interpolation, we assume that $\theta\|\nabla \phi(x)\|\geq 2\sqrt{Ln\epsilon_f}$ which for $\phi(x)$ described by \eqref{eq: ES obj}  implies  $\theta\|\nabla \phi(x)\|\geq 2nL_f$. Plugging these relations into Lemma \ref{lem:prob_bnd_smoothed} gives
\begin{align*}
	N \ge \frac{2n}{\delta \theta^2} +\frac{ (n+2)(n+4)/4+2(n+2)+4}{4\delta n },
\end{align*}

When $n$ is large this bound shows that the number of samples needed to ensure  \eqref{eq.normcond}, with probability $1-\delta$, is  
\vspace{-2pt}
\begin{align*}
	N \geq \frac{2n }{\delta \theta^2}.
\end{align*}

Let us compare the sampling radius $\sigma$ used by the Gaussian smoothing method and that used by the interpolation methods.
For Gaussian smoothing we established that $\sigma=\epsilon_f/\sqrt{n} L_f=\epsilon_f/(\sigma L)$, which implies
$\sigma=\sqrt{\epsilon_f/L}$. For interpolation methods we have $\sigma\approx \theta  \|\nabla \phi(x)\|/ (\sqrt{n}L)$
with  $\theta \|\nabla \phi(x)\| \geq  2\sqrt{Ln\epsilon_f}$, which can give us the lower bound on $\sigma$ as approximately
 $2\sqrt{\epsilon_f/L}$.

%% file: num_exp.tex
The goal of the numerical experiments presented in this section is two-fold: $(1)$ to illustrate empirically the accuracy of different gradient approximations, and $(2)$ to investigate the performance of different gradient approximations within a derivative-free method on reinforcement learning tasks. 

\subsection{Gradient Approximation Accuracy}
First, we compare the numerical accuracy of the gradient approximations obtained by the  following methods: $(1)$ linear interpolation with orthogonal directions (LIOD); $(2)$  linear interpolation with Gaussian directions (LIGD); $(3)$ Gaussian smoothing (GSG); and (4) centered Gaussian smoothing (cGSG) computed via \eqref{eq:cGSG_intro}. The LIOD  and LIDG methods differ in the way the directions $u_i$ are chosen, while both methods solve  \eqref{eq:main_system} to find $g(x)$. For  LIGD the $u_i$'s are chosen to be Gaussian directions and hence the matrix $Q_{\cal X}$ has to be inverted to find $g(x)$. This can be  computationally more costly than using orthonormal $u_i$,  but does not cause large variance in the gradient estimates, hence we include this method in the comparison to emphasize that it too can produce accurate gradient estimates. 
We include cGSG here to show that it does not provide an advantage. 

We measure the relative error $\theta = \tfrac{\| g(x) - \nabla \phi(x)\|}{\| \nabla \phi(x) \|}$
and report the average log of the relative error. Note, for these experiments, we assume that there is no noise, i.e., $\epsilon(x)=0$.

\vspace{-2pt}
\paragraph{Gradient Estimation -- Synthetic Function} We first conduct tests on a synthetic function
\begin{align}
	\phi(x) = \left( \sum_{i=1}^{n/2} M \sin (x_{2i-1}) + \cos(x_{2i})\right) + \frac{L-M}{2n}x^T \pmb{1}_{n \times n}x,
\end{align}
where $n$ is an even number denoting the input dimension, $\pmb{1}_{n \times n}$ denotes an $n$ by $n$ matrix of all ones, and
$L>M>0$. We approximate the gradient of $\phi$ at the origin, for which $\| \nabla \phi(0) \| = \sqrt{\frac{n}{2}}M$. We show results for different $N$ (number of samples) and $\sigma$ (sampling radius) in the boxplots of Figure \ref{fig:fig_grad_approx_sin}.

\begin{figure}[h!]
  \centering
  \includegraphics[trim=40 5 60 15,clip,width=0.49\linewidth]{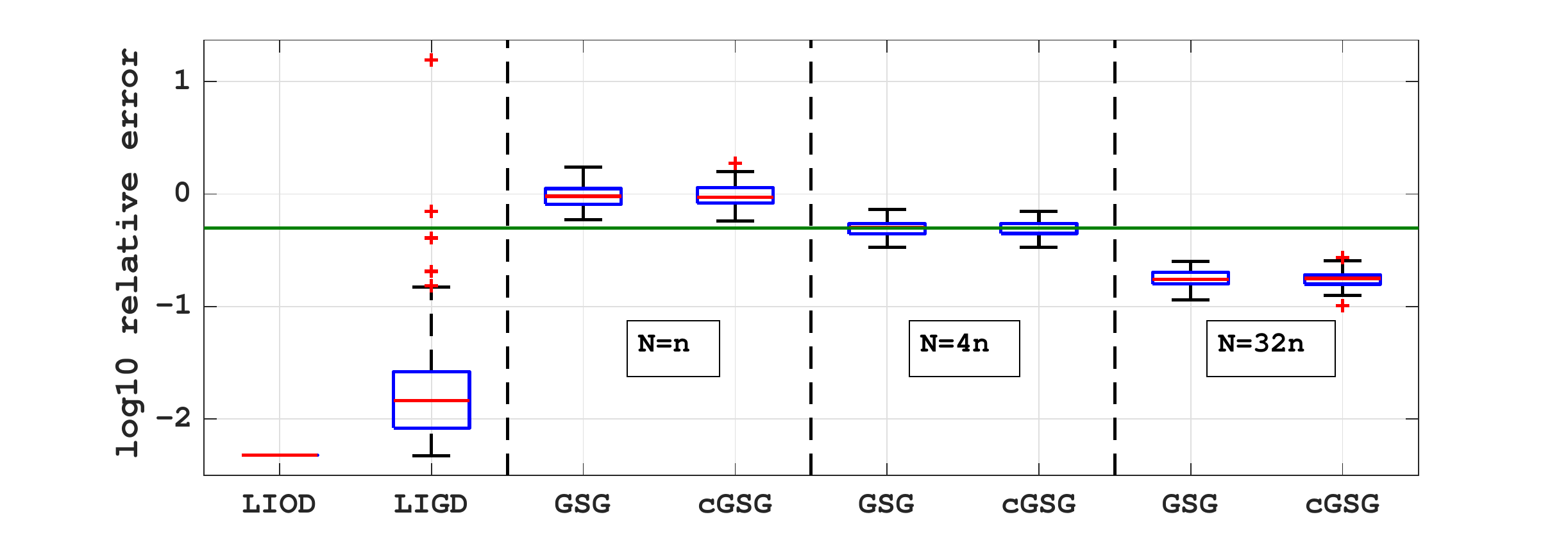}
  \includegraphics[trim=40 5 60 15,clip,width=0.49\linewidth]{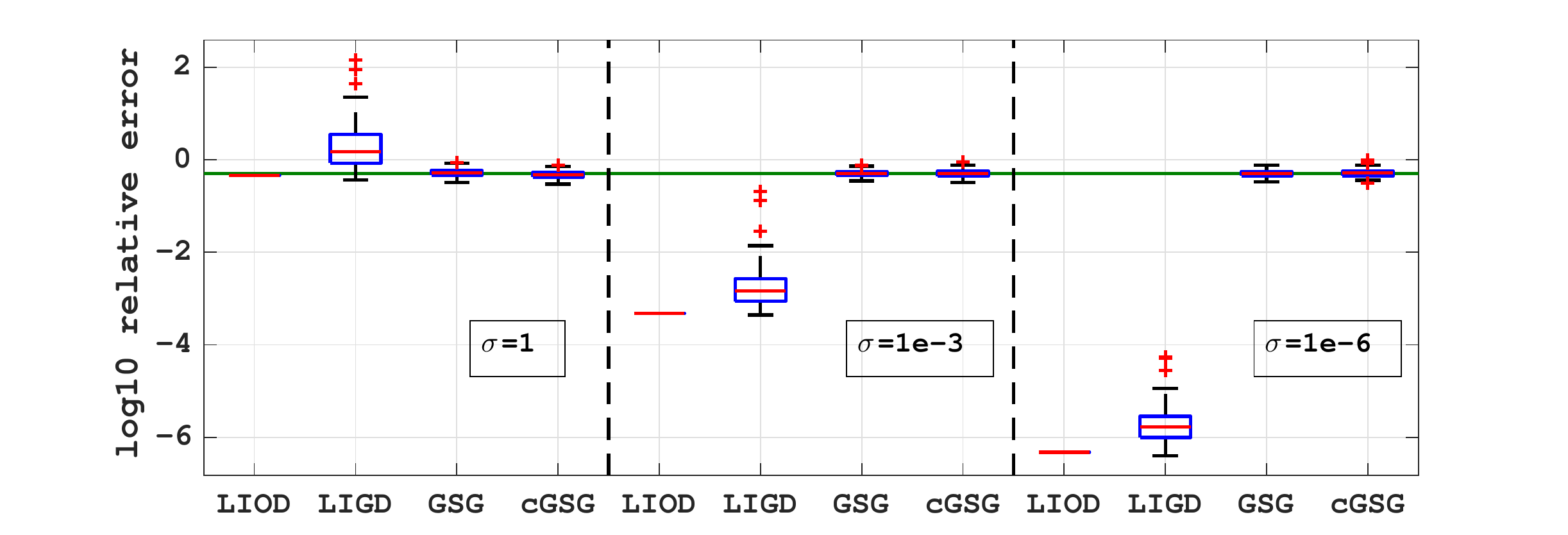}
  \caption{Log of $\theta$ of gradient approximations (LIOD, LIGD,  GSG, cGSG) with different $N$ (left), $\sigma$ (right).}
\label{fig:fig_grad_approx_sin}
\end{figure}

\paragraph{Gradient Estimation -- Schittkowski Functions} 

\begin{wraptable}{R}{0.5\textwidth}
    \begin{minipage}{0.5\textwidth}
    \vspace{-1.25cm}
    \begin{table}[H]
\caption{ Average log of $\theta$ of Gradient Approximations for Schittkowski Problems.}
\scriptsize
\label{tab:num_error}
\centering
\begin{tabular}{ccccc}
\toprule
\textbf{ Approx.} &
 \textbf{$\pmb{ {N}}$} &
 \textbf{$\pmb{ {\sigma}} = 10^{-2}$} &
 \textbf{$\pmb{ {\sigma}} = 10^{-5}$} &
 \textbf{$\pmb{ {\sigma}} = 10^{-8}$}  \\  \midrule
\textbf{LIOD} & $ {n}$ &  {0.2720} &\textbf{-2.6051} &\textbf{-5.3600} \\ \hline
\textbf{LIGD} & $ {n}$ &  {0.7757} &\textbf{-2.0992} &\textbf{-4.7714}  \\ \hline
\textbf{GSG} & $ {n}$ &  {0.7449} &  {0.0979} &{-0.0254} \\ \hdashline
		&$2 {n}$ &  {0.6527} &{-0.0220} &{-0.1529}  \\ \hdashline
	 	&$4 {n}$ &  {0.5309} &{-0.1580} &{-0.2988} \\ \hdashline
	 	&$8 {n}$ &  {0.4020} &\textbf{-0.3023} &\textbf{-0.4486}  \\ \hline
 \textbf{cGSG}  &$2 {n}$ &  {0.1219} &{-0.0453} &{-0.0403} \\ \hdashline
	 & $4 {n}$ &  {0.0022} &{-0.1761} &{-0.1796} \\ \hdashline
	 & $8 {n}$ &{-0.1159} &\textbf{-0.3209} &\textbf{-0.3136} \\ \hdashline
	 & $16 {n}$ &{-0.2387} &\textbf{-0.4649}  &\textbf{-0.4598} \\ \bottomrule
\end{tabular}
\end{table}
    \end{minipage}
  \end{wraptable}    
Next, we test the different gradient approximations on the 69 functions from the Schittkowski test set \cite{schittkowski2012more}.  For each problem we generated 100 points and computed the  gradient approximations. Table \ref{tab:num_error} summarizes the results of these experiments. We show the average log of the relative error for different choices of $\sigma$ ($\sigma \in \{ 10^{-2}, 10^{-5}, 10^{-8}\}$), and where appropriate different choices of $N$ (number of samples). Bold values indicate values of $\theta < \sfrac{1}{2}$.

\subsection{Reinforcement Learning}

In this section, we present numerical results for reinforcement learning (RL) tasks from the $\mathrm{OpenAI}$ $\mathrm{Gym}$ library \citep{brockman2016openai}. We compare LIOD with fixed $\alpha_k = \alpha$ and with $\alpha_k$ chosen via a line search, GSG with fixed $\alpha_k= \alpha$ and also the standard finite difference method (FD), which is a particular version of LIOD, with fixed $\alpha_k= \alpha$.

In all RL experiments the blackbox function $f$ takes as input the parameters of the policy $\pi_{x}:\mathcal{S} \rightarrow \mathcal{A}$ which maps states $\mathcal{S}$ to proposed actions $\mathcal{A}$. The output of $f$ is the total reward obtained by an agent applying that particular policy $\pi_{x}$ in the given environment. To encode policies $\pi_{x}$, we used fully-connected feedforward neural networks with two hidden layers, each with $h=41$ neurons and with $\mathrm{tanh}$ nonlinearities. The matrices of connections were encoded by low-displacement rank neural networks, as in several recent papers on applying orthogonal directions in gradient estimation for derivative free methods in reinforcement learning; see \cite{choro2}. We did not apply any additional techniques such as state/reward renormalization, ranking or filtering, in order to solely focus on the evaluation of the presented methods.

In order to construct orthogonal samples, at each iteration we conducted orthogonalization of random Gaussian matrices, with i.i.d entries sampled from $\mathcal{N}(0,1)$, via a Gram-Schmidt procedure; see \cite{choro2}. We should note that in the case of large $n$ instead of the orthogonalized Gaussian matrices, we can use constructions where orthogonality is embedded into the structure, such as random Hadamard matrices  \cite{choro2}, thus reducing the computational cost of generating random orthogonal matrices from  $\mathcal{O}(n^3)$ to $\mathcal{O}(n\log n)$. Note that the use of random Hadamard matrices introduces a small bias, however, they have been shown to work well in practice. 

All experiments were run with hyperparameter $\sigma=0.1$. Methods that did not apply a line search were run using the $\mathrm{Adam}$ optimizer \cite{kingma2014adam} with $\alpha = 0.01$. For the line search experiments, that adaptive $\alpha$ was chosen via the Armijo condition with Armijo parameter $c_{1}=0.2$ and backtracking factor $\tau=0.3$. For each environment and each method we ran $k=3$ experiments corresponding to different random seeds.

In Figure \ref{fig.RL}, we show the average (solid lines) and max/min (dashed lines) over a number of runs. While our theory is the same for FD and LIOD, our experiments show that for these tasks, choosing $u_i$ to be orthonormal but random helps the algorithm to avoid getting stuck in local maxima. We observe that the LIOD method is superior to the GSG and that line-search provides some improvements over $\mathrm{Adam}$.

\begin{figure}[h!]
  \centering
  \includegraphics[trim=25 175 60 190,clip,width=0.32\linewidth]{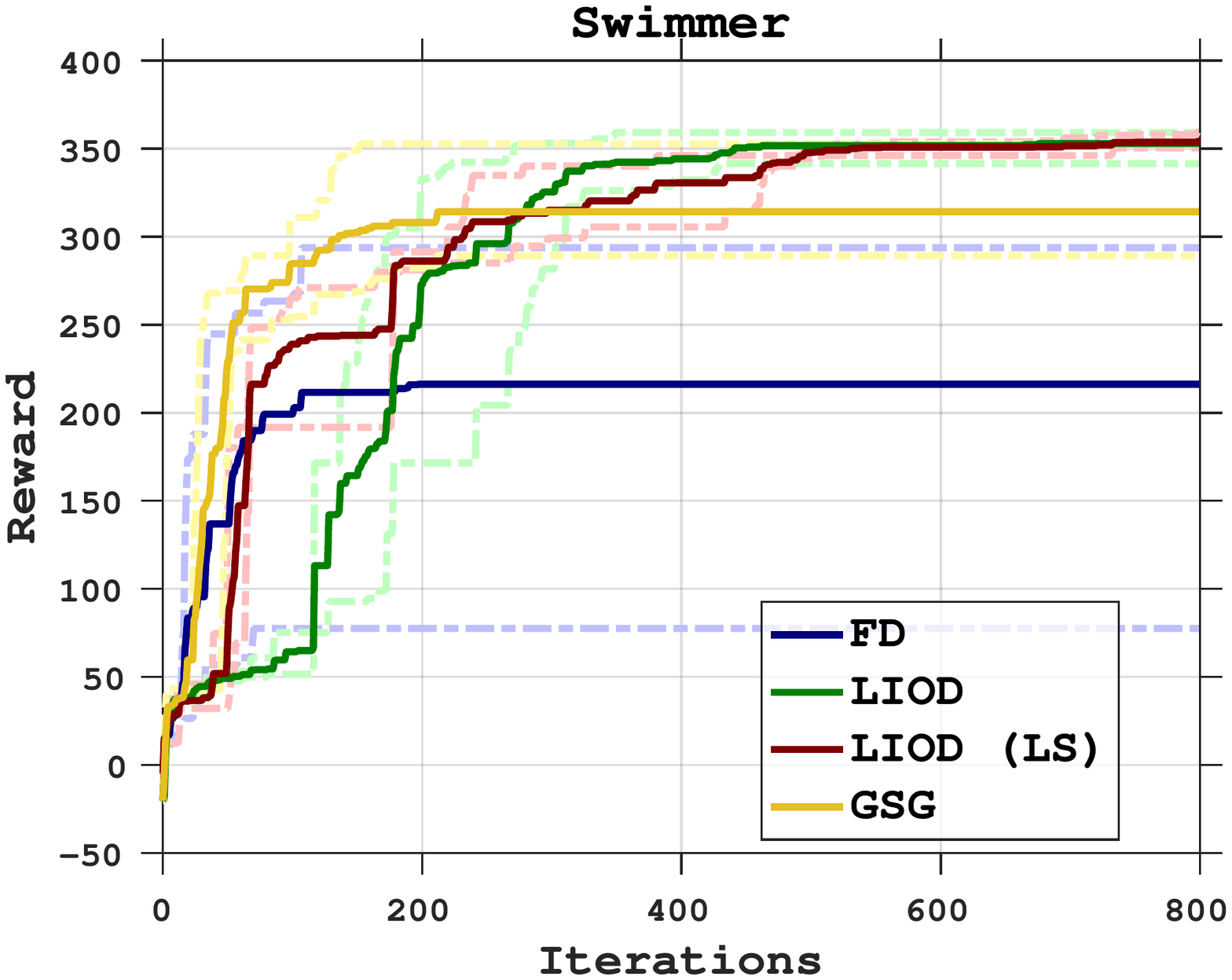}
  \includegraphics[trim=25 175 60 190,clip,width=0.32\linewidth]{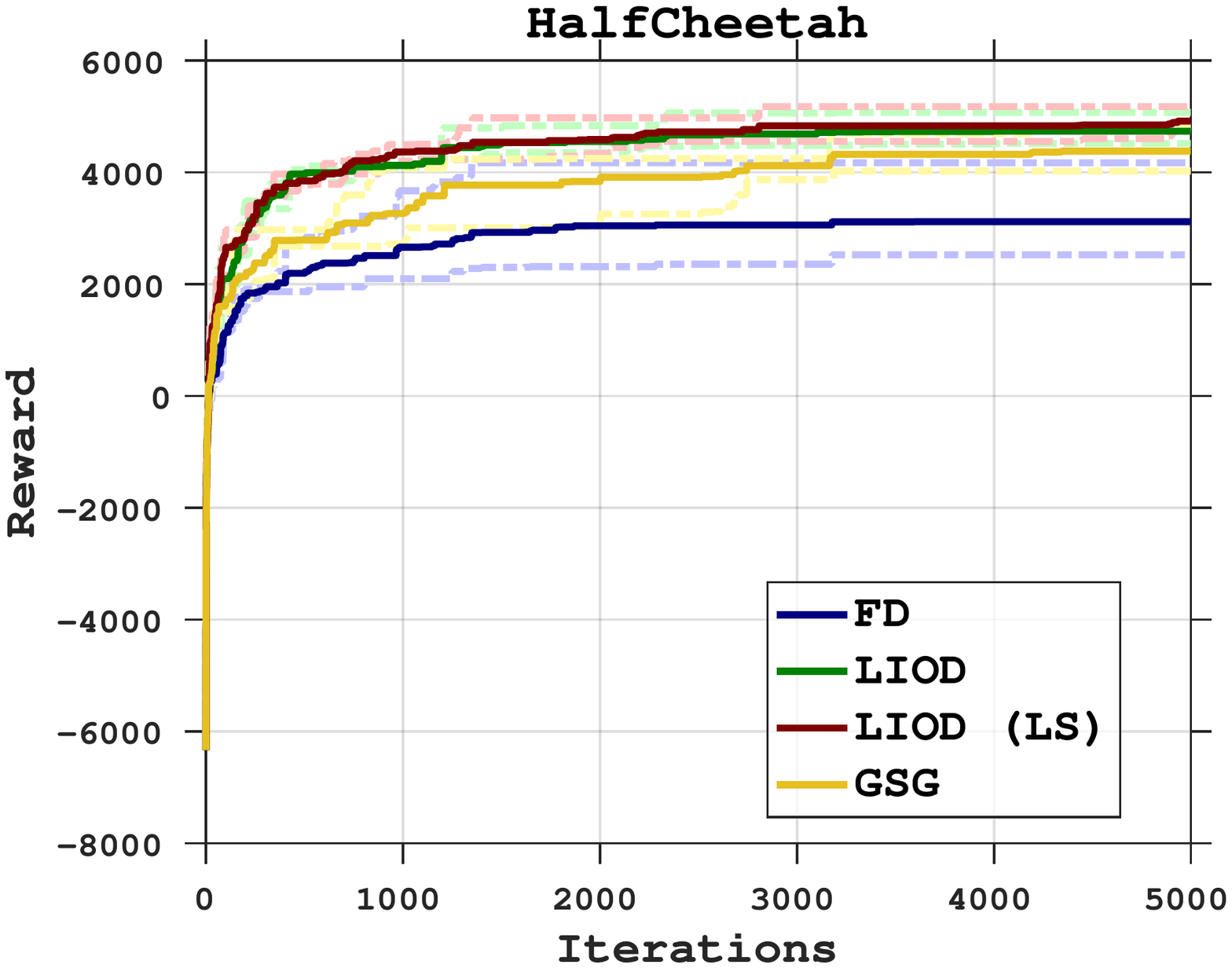}
\includegraphics[trim=25 175 60 190,clip,width=0.32\linewidth]{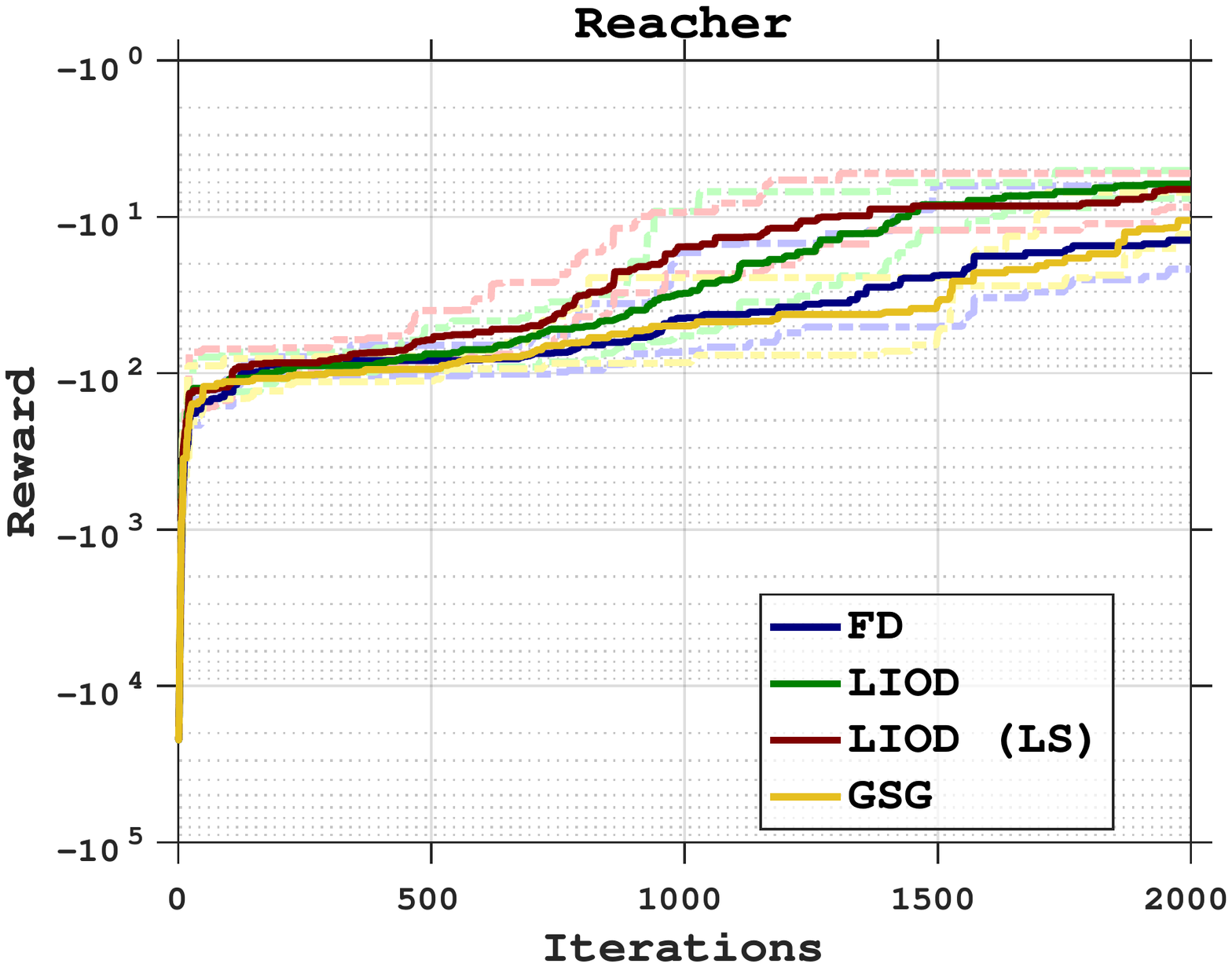}
      \caption{Reinforcement learning tasks: \texttt{Swimmer} (left), \texttt{HalfCheetah} (center), \texttt{Reacher} (right).}
\label{fig.RL}
\end{figure}

%% file: final_remarks.tex
This paper describes and analyzes a  line-search derivative-free optimization algorithm adapted to the case of bounded noise. Complexity bounds are derived, in terms of convergence to a neighborhood of the optimal solution, under certain conditions on the gradient approximations. It is shown that these conditions can be satisfied by two popular methods for approximating gradients, with one method having a clear advantage over the other. Empirical tests on synthetic problems and on reinforcement learning tasks support the theoretical findings.